\documentclass{article}

\usepackage{amssymb,latexsym,amsmath,amsthm,amsfonts,graphics}
\usepackage{graphicx}

\usepackage{caption}
\usepackage{subcaption}
\usepackage[rightcaption]{sidecap}

\usepackage{color}
\usepackage{lineno}
\usepackage{multirow}
\usepackage{epstopdf}
\usepackage{rotating}
\usepackage{cite}

\usepackage[a4paper, total={6.8in, 9in}]{geometry}
\usepackage{hyperref}
\usepackage{tikz}

\newtheorem{theorem}{Theorem}[section]

\newtheorem{lemma}{Lemma}[section]

\newtheorem{corollary}{Corollary}[section]
\newtheorem{definition}{Definition}[section]
\newtheorem{example}{Example}[section]

\makeatletter
\def\ps@pprintTitle{%
 \let\@oddhead\@empty
 \let\@evenhead\@empty
 \def\@oddfoot{\centerline{\thepage}}%
 \let\@evenfoot\@oddfoot}
\makeatother

\begin{document}

\begin{center}
{\bf {\Large Linear Diophantine Graphs}}
\end{center}

\begin{center}
{ \bf A. Nasr*$^1$, \ M. Anwar*$^2$, \ M. A. Seoud*$^3$,   \ A. Elsonbaty*$^4$}
\vspace{3mm}\\
 *Department of Mathematics, Faculty of Science, Ain Shams University, 11566, Abbassia, Cairo, Egypt.
\vspace{3mm}\\
e-mails: $^1$ \ \href{mailto:amr_fatouh@sci.asu.edu.eg}{\url{amr_fatouh@sci.asu.edu.eg}},\hspace{0.2cm} $^2$ \ \href{mailto:mohamedanwar@sci.asu.edu.eg}{\url{mohamedanwar@sci.asu.edu.eg}},\\
\hspace{1.5cm}$^3$ \ \href{mailto:m.a.seoud@sci.asu.edu.eg}{\url{m.a.seoud@sci.asu.edu.eg}},\quad  $^4$ \ \href{mailto:ahmedelsonbaty@sci.asu.edu.eg}{\url{ahmedelsonbaty@sci.asu.edu.eg}}
\end{center}


\begin{abstract}
This manuscript introduces linear Diophantine labelling, a new method for assigning labels to the vertices of finite, simple, undirected graphs. A key feature of this method is a divisibility condition imposed on the edges, incorporating number-theoretic properties into graph labelling. The study focuses on identifying maximal graphs that admit such labellings and computes their number of edges and the degree of each vertex. Number-theoretic techniques are employed to examine structural properties, including the characterization of maximum degree vertices and conditions for nonadjacent vertices. The manuscript also establishes necessary and sufficient conditions for vertices with equal degrees, offering new insights into the interaction between graph theory and number theory.
\end{abstract}
\par
\bigskip \noindent Keywords: Graph labelling, Prime graph, Diophantine graph, Equal degrees, $p$-adic valuation.
\par
\bigskip \noindent AMS Subject Classification: 05C07, 05C78, 11A05, 11A25, 11D88.


\section{Introduction}

\hspace{0.5cm} A graph labelling is an assignment of real values to the vertices or edges or both edges and vertices that meet some conditions. It constitutes a significant area within graph theory and has numerous applications, including in coding theory, radar systems, astronomy, missile guidance, communication networks, cryptography, network security, X-ray crystallography, and database management. In this manuscript, we investigate a new type of labeling, referred to as linear Diophantine labeling, which constitutes a generalization of prime labeling. Labeling problems are pivotal in studying computational complexity and algorithmic graph theory, serving as a bridge between number theory and graph theory. The concept of prime labelling was first introduced by R. Entringer and later formally defined in a paper by R. Tout \cite{tout}. A simple graph $G=(V, E)$ with $n$ vertices is said to have a prime labelling if there is a bijective function $f$ mapping $V$ to $\{1, 2, \dots, n\}$ in which for every $uv\in E$, the values $f(u),f(v)$ are relatively prime, where $V$ is the vertex set and $E$ is the edge set. Many results of prime labelling and other types of prime labelling can be found in the dynamic survey of graph labelling by J.A. Gallian \cite{Gallian}. Other variants of prime labelling have been introduced by many researchers (see, for example, \cite{2}, \cite{3}, \cite{7}, \cite{15}).

This paper gives a generalization for the concept of prime graphs. More precisely, a simple graph $G$ with $n$ vertices is defined as a linear Diophantine graph (or simply, a Diophantine graph) if there is a bijective function $f$ mapping $V$ to $\{1, 2, \dots, n\}$ in which for every $uv\in E$, the greatest common divisor of $f(u)$ and $f(v)$ divides $n$. Clearly, any prime graph is a Diophantine graph. Moreover, any Diophantine graph with a prime number of vertices is a prime graph. In a maximal Diophantine graph with $n$ vertices, an explicit formula of the number of edges and the degree of a vertex are proved, the full degree vertices are determined. Furthermore, we provide sufficient and necessary conditions that determine the equality of degrees among vertices whose degrees are smaller than $n-1$ in a maximal Diophantine graphs.

The paper is organized as follows. Section 2 includes some definitions, preliminaries, lemmas and direct observations. Section 3 presents the number of edges formula in a maximal Diophantine graph with $n$ vertices. Section 4 is split into three subsections, each unveiling important results related to a maximal Diophantine graph with order $n$. Subsection 4.1 covers a characterization of labels of vertices with full degree and a formula that computes the degrees of any vertex. Subsection 4.2 and 4.3 encompass some necessary and sufficient conditions for equality of degrees.

For definitions and terminologies of graph theory, we follow F. Harary \cite{Harary} and A. Bickle \cite{Bickle}. Also, we follow D. Burton \cite{Burton} and K.H. Rosen \cite{Rosen} for basic definitions and notations in number theory.

\begin{theorem}\cite{Burton}, \cite{Rosen}\label{FTA} (Fundamental Theorem of Arithmetic)\\
  Every positive integer $n>1$ can be expressed as a product of primes; this representation is unique, apart from the order in which the factors occur.
\end{theorem}
\begin{definition}\cite{Bickle}
  Two graphs $G_1, G_2$ are said to be isomorphic, denoted by $G_1\cong G_2$, if there exists a bijective map $\varphi: V(G_1)\rightarrow V(G_2)$ such that for all $u,v\in V(G_1)$,
  $$uv\in E(G_1)\quad \mbox{if and only if}\quad \varphi(u)\varphi(v)\in E(G_2).$$
\end{definition}
\section{Definitions and Preliminaries}
In the beginning of this section, we present the definitions of the prime graphs and the Diophantine graphs and we explore the relation between them. Additionally, we provide direct observations on several number-theoretic lemmas.

\begin{definition}\cite{Seoud1}, \cite{tout}
  A simple graph $G$ of order $n$ is called a prime graph if there exists a bijective mapping $f: V\rightarrow \{1, 2, \dots, n\}$ such that $(f(u),f(v))=1$ for all $uv\in E$. The function $f$ is called a prime labelling of $G$. A prime graph of order $n$ is called a maximal prime graph, denoted by $(R_n,f)$, if any new edge added to this graph turns it into a non-prime graph.
\end{definition}

\begin{definition}
  A simple graph $G$ of order $n$ is called a linear Diophantine graph (or simply, a Diophantine graph) if there exists a bijective mapping $f:V\rightarrow \{1, 2, \dots, n\}$ such that  $(f(u),f(v))\mid n$ for all $uv\in E$. The function $f$ is called a Diophantine labelling of $G$. A Diophantine graph of order $n$ is called a maximal Diophantine graph, denoted by $(D_n,f)$, if any new edge added to this graph turns it into a non-Diophantine graph. If there is no ambiguity, we omit $f$ from $(D_n,f)$ and simply write $D_n$.
\end{definition}

\begin{example}
The Peterson graph $P$ is a Diophantine graph, but the complete graph $K_5$ is not.
\end{example}
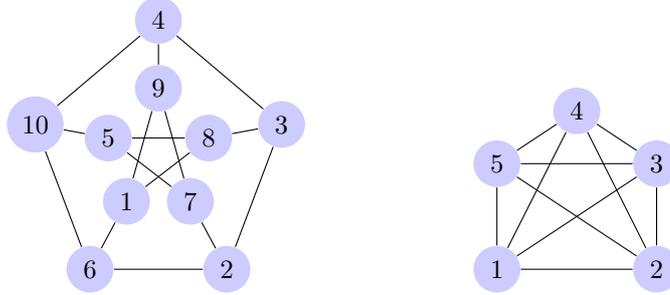
\begin{figure*}[h!]

\centering
\begin{subfigure}{0.3\textwidth}
 \centering
\begin{tikzpicture}
  [scale=0.6,auto=center,every node/.style={circle,fill=blue!20}]
  \node (v1)  at (1,0)         {$6$};
  \node (v2)  at (4,0)         {$2$};
  \node (v3)  at (5.2,3.2)     {$3$};
  \node (v4)  at (2.5,5.5)     {$4$};
  \node (v5)  at (-0.2,3.2)    {$10$};
  \node (v6)  at (1.8,1.5)     {$1$};
  \node (v7)  at (3.2,1.5)     {$7$};
  \node (v8)  at (3.6,2.9)     {$8$};
  \node (v9)  at (2.5,4)       {$9$};
  \node (v10) at (1.4,2.9)     {$5$};
  \draw (v1) -- (v2);
  \draw (v1) -- (v6);
  \draw (v1) -- (v5);

  \draw (v2) -- (v7);
  \draw (v2) -- (v3);

  \draw (v3) -- (v8);
  \draw (v3) -- (v4);

  \draw (v4) -- (v9);
  \draw (v4) -- (v5);

  \draw (v5) -- (v10);

  \draw (v6) -- (v8);
  \draw (v6) -- (v9);

  \draw (v7) -- (v9);
  \draw (v7) -- (v10);

  \draw (v8) -- (v10);
\end{tikzpicture}
\end{subfigure}
~
\begin{subfigure}{0.3\textwidth}
\centering
\begin{tikzpicture}
  [scale=0.6,auto=center,every node/.style={circle,fill=blue!20}]
  \node (v1) at (0,0)    {$1$};
  \node (v2) at (3,0)    {$2$};
  \node (v3) at (3,2)    {$3$};
  \node (v4) at (1.5,3)  {$4$};
  \node (v5) at (0,2)    {$5$};
  \draw (v1) -- (v2);
  \draw (v1) -- (v3);
  \draw (v1) -- (v4);
  \draw (v1) -- (v5);

  \draw (v2) -- (v3);
  \draw (v2) -- (v4);
  \draw (v2) -- (v5);

  \draw (v3) -- (v4);
  \draw (v3) -- (v5);

  \draw (v4) -- (v5);
 \end{tikzpicture}
\end{subfigure}
\captionof{figure}{$P$ is a Diophantine graph while $K_5$ is not}\label{figure0}
\end{figure*}

\begin{example}
The following are some maximal Diophantine graphs.
\end{example}
\begin{figure*}[h!]

\centering
\begin{subfigure}{0.13\textwidth}
 \centering
\begin{tikzpicture}
  [scale=0.6,auto=center,every node/.style={circle,fill=blue!20}]
  \node (v1) at (0,0)    {$1$};
  \node (v2) at (3,0)    {$2$};
  \node (v3) at (3,3)    {$3$};
  \node (v4) at (0,3)    {$4$};
\draw (v1) -- (v2);
\draw (v1) -- (v3);
\draw (v1) -- (v4);
\draw (v2) -- (v3);
\draw (v2) -- (v4);
\draw (v3) -- (v4);
\end{tikzpicture}
\end{subfigure}
~
\begin{subfigure}{0.13\textwidth}
\begin{tikzpicture}
  [scale=0.6,auto=center,every node/.style={circle,fill=blue!20}]
  \node (v1)                     at (0,0)      {$1$};
  \node (v2)[circle,fill=red!20] at (3,0)      {$2$};
  \node (v3)                     at (3,2.5)    {$3$};
  \node (v4)[circle,fill=red!20] at (1.5,4)    {$4$};
  \node (v5)                     at (0,2.5)    {$5$};
  \draw (v1) -- (v2);
  \draw (v1) -- (v3);
  \draw (v1) -- (v4);
  \draw (v1) -- (v5);

  \draw (v2) -- (v3);
  \draw (v2) -- (v5);

  \draw (v3) -- (v4);
  \draw (v3) -- (v5);

  \draw (v4) -- (v5);
\end{tikzpicture}
\end{subfigure}
~
\begin{subfigure}{0.16\textwidth}
 \centering
\begin{tikzpicture}
  [scale=0.6,auto=center,every node/.style={circle,fill=blue!20}]
  \node (v1) at (0,0)      {$1$};
  \node (v2) at (1.7,0)    {$2$};
  \node (v3) at (2.7,2)    {$3$};
  \node (v4) at (1.7,4)    {$4$};
  \node (v5) at (0,4)      {$5$};
  \node (v6) at (-1,2)     {$6$};
  \draw (v1) -- (v2);
  \draw (v1) -- (v3);
  \draw (v1) -- (v4);
  \draw (v1) -- (v5);
  \draw (v1) -- (v6);

  \draw (v2) -- (v3);
  \draw (v2) -- (v4);
  \draw (v2) -- (v5);
  \draw (v2) -- (v6);

  \draw (v3) -- (v4);
  \draw (v3) -- (v5);
  \draw (v3) -- (v6);

  \draw (v4) -- (v5);
  \draw (v4) -- (v6);

  \draw (v5) -- (v6);
\end{tikzpicture}
\end{subfigure}
~
 \begin{subfigure}{0.18\textwidth}
 \centering
   \begin{tikzpicture}
   [scale=0.6,auto=center,every node/.style={circle,fill=blue!20}]
  \node (v1) at (1,6)   {$1$};
  \node (v5) at (-1,6)  {$5$};
  \node (v7) at (-2,4)  {$7$};
  \node (v3) at (2,4)   {$3$};

  \node (v2)[circle,fill=red!20] at (-1.5,2)  {$2$};
  \node (v6)[circle,fill=red!20] at (0,1.5)   {$6$};
  \node (v4)[circle,fill=red!20] at (1.5,2)   {$4$};
  \draw (v1) -- (v2);
  \draw (v1) -- (v3);
  \draw (v1) -- (v4);
  \draw (v1) -- (v5);
  \draw (v1) -- (v6);
  \draw (v1) -- (v7);

  \draw (v5) -- (v2);
  \draw (v5) -- (v3);
  \draw (v5) -- (v4);
  \draw (v5) -- (v6);
  \draw (v5) -- (v7);

  \draw (v7) -- (v2);
  \draw (v7) -- (v3);
  \draw (v7) -- (v4);
  \draw (v7) -- (v6);

  \draw (v3) -- (v2);
  \draw (v3) -- (v4);
   \end{tikzpicture}
  \end{subfigure}
  ~
  \begin{subfigure}{0.2\textwidth}
 \centering
\begin{tikzpicture}
  [scale=0.5,auto=center,every node/.style={circle,fill=blue!20}]
  \node                     (v1) at (0,0)     {$1$};
  \node                     (v2) at (2,0)     {$2$};
  \node[circle,fill=red!20] (v3) at (3.5,2)   {$3$};
  \node                     (v4) at (3.5,4)   {$4$};
  \node                     (v5) at (0,6)     {$5$};
  \node[circle,fill=red!20] (v6) at (2,6)     {$6$};
  \node                     (v7) at (-1.5,4)  {$7$};
  \node                     (v8) at (-1.5,2)  {$8$};
  \draw (v1) -- (v2);
  \draw (v1) -- (v3);
  \draw (v1) -- (v4);
  \draw (v1) -- (v5);
  \draw (v1) -- (v6);
  \draw (v1) -- (v7);
  \draw (v1) -- (v8);

  \draw (v2) -- (v3);
  \draw (v2) -- (v4);
  \draw (v2) -- (v5);
  \draw (v2) -- (v6);
  \draw (v2) -- (v7);
  \draw (v2) -- (v8);

  \draw (v3) -- (v4);
  \draw (v3) -- (v5);
  \draw (v3) -- (v7);
  \draw (v3) -- (v8);

  \draw (v4) -- (v5);
  \draw (v4) -- (v6);
  \draw (v4) -- (v7);
  \draw (v4) -- (v8);

  \draw (v5) -- (v6);
  \draw (v5) -- (v7);
  \draw (v5) -- (v8);

  \draw (v6) -- (v7);
  \draw (v6) -- (v8);

  \draw (v7) -- (v8);
\end{tikzpicture}
 \end{subfigure}
\captionof{figure}{Maximal Diophantine Graphs $D_4,\dots, D_{8}$}
\end{figure*}
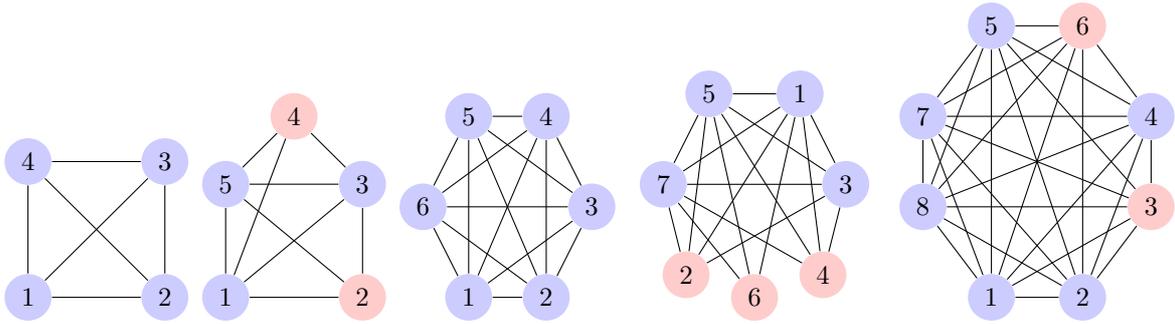
It is easy to see that every prime graph is a Diophantine graph, but the converse is not necessarily true, e.g., $D_4,D_6,D_8$.

%

\begin{theorem}\label{prime-dioph1}
  Let $n>1$ be a prime number. A graph $G$ of order $n$ is a Diophantine graph if and only if $G$ is a prime graph.
\end{theorem}
\begin{proof}
  The proof is clear.
\end{proof}

\begin{theorem}\label{prime-dioph2}
 For every $n>1$, $D_n\cong R_n$ if and only if $n$ is a prime number.
\end{theorem}
\begin{proof}
Let $K_n$ be the complete graph of order $n$ and $f: V(K_n)\rightarrow \{1, 2, \dots, n\}$ be a bijective map. Define two sets
$$E(R_n):=\{uv\in E(K_n):(f(u),f(v))=1\}$$
and
$$E(D_n):=\{uv\in E(K_n):(f(u),f(v))\mid n\}.$$
Let $D_n\cong R_n$. Then $E(R_n)=E(D_n)$. Therefore,
$$(f(u),f(v))\mid n \quad\mbox{if and only if}\quad (f(u),f(v))=1.$$
Suppose, by the way of contradiction, $n>1$ is a composite number. Then there exists a prime number $p\mid n$ such that $p \leq\frac{n}{2}$. Let $f(u)=p\leq n$ and $f(v)=2p\leq n$. Then $(f(u),f(v))=p\mid n$ and $(f(u),f(v))\neq1$ which contradicts $(f(u),f(v))=1$. Hence, $n$ is a prime number.

Conversely, let $n$ be a prime number and let $uv\in E(R_n)$. Then $(f(u),f(v))=1$. Consequently, $(f(u),f(v))\mid n$. Therefore, $uv\in E(D_n)$. On the other hand, let $uv\in E(D_n)$. Then $(f(u),f(v))\mid n$. Since $n$ is a prime number, we have either $(f(u),f(v))=1$ or $(f(u),f(v))=n$. In the case that $(f(u),f(v))=1$, we have nothing to proof. In the other case that $(f(u),f(v))=n,$ we get $f(u)=f(v)=n$ which is a contradiction with the assumption. Consequently, $(f(u),f(v))=1$. Therefore, $uv\in E(R_n)$. Hence, $D_n\cong R_n$.
\end{proof}

In the subsequent sections, the $p$-adic valuation of $n\in\mathbb{Z}^+$ is defined by $v_p(n):=t$ such that $p^t\mid n$ and $p^{t+1}\nmid n$, where $p$ is a prime number and $t\in\mathbb{N}:=\mathbb{Z}^+\cup \{0\}$ and $\mathbb{Z}^+$ is the set of positive integers. The set of all multiple of $a\in\mathbb{Z}^+$ up to $n$ is denoted by $M_a$, i.e. $M_a:=\left\{ka \ : \ k=1,2,\dots,\left\lfloor\frac{n}{a}\right\rfloor\right\}.$ In any graph $G$ of order $n$, we call $f^*(u):=\frac{f(u)}{(f(u), n)}$ the reduced label of a vertex $u$ and the neighborhood of $s\in V(G)$ is defined by $N(s):=\{u\in V(G) :su\in E(G)\}$. Now, we present a crucial definition and important number-theoretic lemmas that will allow us to formulate and prove our results smoothly.

\begin{definition}
  A number $p^{\acute{v}_p(n)}$ is called a critical prime power number with respect to a prime number $p$ and a positive integer $n$, where $\acute{v}_p(n):=v_p(n)+1$ is called the successor of the $p$-adic valuation.
\end{definition}

\begin{lemma}\label{rem1} Let $u,v$ be two vertices in $D_n$ and $f$ be a Diophantine labelling of $D_n$. Then
\begin{enumerate}
    \item[i.] for every prime number $p$, $p\mid f^*(u)$ if and only if $p^{\acute{v}_p(n)}\mid f(u);$
    \item[ii.] if $f(u)\mid f(v)$, then $f^*(u)\mid f^*(v).$
\end{enumerate}
\end{lemma}
\begin{proof}
The proof of part $i.$ is straightforward. Let us prove the second part. The $p$-adic valuation function \cite{Bachman} is extended normally to rational numbers as follows
  \begin{equation*}
      v_p\left(\frac{b}{a}\right):=v_p(b)-v_p(a),
  \end{equation*}
for any $a,b\in\mathbb{Z}^+$ and for every prime number $p$. Let $u,v$ be two vertices in $V(D_n)$ such that $f(u)\mid f(v).$ Using the basic properties of the $p$-adic valuation function, we have
 \begin{equation*}
    \begin{split}
       v_p\left(\frac{f(v)}{f(u)}\frac{(f(u), n)}{(f(v), n)}\right) & =v_p(f(v))-v_p(f(u))+v_p\big((f(u), n)\big)-v_p\big((f(v), n)\big). \\
                                                                    & =v_p(f(v))-v_p(f(u))+\min\big\{v_p(f(u)), v_p(n)\big\}-\min\big\{v_p(f(v)), v_p(n)\big\}.
    \end{split}
 \end{equation*}
Therefore, one can easily check that
$$v_p\left(\frac{f(v)}{f(u)}\frac{(f(u), n)}{(f(v), n)}\right)\geq0$$
in all possible cases between $v_p(f(u)),v_p(n)$ and $v_p(f(v))$, namely
$$v_p(n)<v_p(f(u)), \ v_p(f(u))\leq v_p(n)\leq v_p(f(v)) \ \mbox{and} \ v_p(f(v))<v_p(n).$$
Therefore,
$$\frac{f(u)}{(f(u), n)} \ \Bigg| \ \frac{f(v)}{(f(v), n)}.$$
Hence, the proof follows.
\end{proof}

We will use the following identity:
$\frac{1}{j}-\frac{1}{j+1}=\left(1+\frac{2}{j}\right)\left[\frac{1}{j+1}-\frac{1}{j+2}\right]$, $j\in\mathbb{Z}^+$
for the proof of Lemma \ref{lem1.13}.
\begin{lemma}\label{lem1.13}
  For every $n\in\mathbb{Z}^+$, there exists $j\in \left\{ \lfloor\sqrt{n}\rfloor,  \lfloor\sqrt{n}\rfloor+1, \ldots, n \right\}$ such that
 \begin{equation*}
   \frac{n}{j+1} < \lfloor\sqrt{n}\rfloor \leq \frac{n}{j} \quad \mbox{and} \quad \frac{n}{j}-\frac{n}{j+1}<1.
 \end{equation*}
\end{lemma}
\begin{proof}
Clearly, we have
\begin{equation*}
    \begin{cases}
\lfloor\sqrt{n}\rfloor \in \left(\frac{n}{\lfloor\sqrt{n}\rfloor +1}, \frac{n}{\lfloor\sqrt{n}\rfloor}\right] \quad \mbox{and}\quad \frac{n}{\lfloor\sqrt{n}\rfloor}-\frac{n}{\lfloor\sqrt{n}\rfloor +1} <1, & \quad \mbox{if} \quad n < \lfloor\sqrt{n}\rfloor(\lfloor\sqrt{n}\rfloor+1); \\
&\\
 \lfloor\sqrt{n}\rfloor \notin \left(\frac{n}{\lfloor\sqrt{n}\rfloor+1}, \frac{n}{\lfloor\sqrt{n}\rfloor }\right] \quad\mbox{and}\quad \frac{n}{\lfloor\sqrt{n}\rfloor}-\frac{n}{\lfloor\sqrt{n}\rfloor +1}\geq1, & \quad \mbox{if} \quad n \geq \lfloor\sqrt{n}\rfloor(\lfloor\sqrt{n}\rfloor+1).\\
\end{cases}
\end{equation*}
Since $\lfloor\sqrt{n}\rfloor <\frac{n}{\lfloor\sqrt{n}\rfloor}$ and $\frac{n}{j}-\frac{n}{j+1}>\frac{n}{j+1}-\frac{n}{j+2},$ therefore there exists $j\geq \lfloor\sqrt{n}\rfloor$ such that
\begin{equation*}
  \lfloor\sqrt{n}\rfloor \in\left(\frac{n}{j+2}, \frac{n}{j+1}\right] \quad \mbox{and} \quad \frac{n}{j+1}-\frac{n}{j+2}\leq\frac{n}{\lfloor\sqrt{n}\rfloor+1}-\frac{n}{\lfloor\sqrt{n}\rfloor+2}<1,
\end{equation*}
which complete the proof.
\end{proof}

\begin{lemma}\label{lem1.14}
 Let $a, b$ and $n\in \mathbb{Z}^+$. If $\lfloor\frac{n}{a}\rfloor=\lfloor\frac{n}{b}\rfloor$ and $a\neq b$, then $ab>n$.
\end{lemma}
\begin{proof}
Clearly, the statement is true in case of positive integers $a$ or $b$ are greater than or equal to $n\in \mathbb{Z}^+$. Now, let $a,b<n$ such that $\lfloor\frac{n}{a}\rfloor=\lfloor\frac{n}{b}\rfloor$. Therefore, there exists $j\in\{1,2, \dots, n\}$ such that $\frac{n}{j+1}<a, b\leq\frac{n}{j}$. Consequently, $\frac{n}{j}-\frac{n}{j+1}> 1$, since $a \neq b$. Using Lemma \ref{lem1.13}, we get that there exists $i\in\left\{\lfloor\sqrt{n}\rfloor,\lfloor\sqrt{n}\rfloor+1, \dots, n\right\}$ such that
\begin{equation*}
 \frac{n}{i+1} < \lfloor\sqrt{n}\rfloor \leq \frac{n}{i} \quad \mbox{and} \quad \frac{n}{i}-\frac{n}{i+1}<1.
\end{equation*}
Thus, $a,b>\lfloor\sqrt{n}\rfloor$. Hence, $ab> n$.
\end{proof}

\begin{lemma}\label{lem1.15}
 Let $a,b,n\in \mathbb{Z}^+$ and let $p_1,p_2,q_1,q_2$ be distinct prime numbers and $\alpha_1,\alpha_2,\beta_1$, $\beta_2$ be positive integers. If $\frac{a}{(a,n)}=p_1^{\alpha_1} p_2^{\alpha_2}$ and $\frac{b}{(b,n)}= q_1^{\beta_1} q_2^{\beta_2}$ such that
\begin{equation}\label{inequality1}
 \frac{n}{i+1} < p_1^{\acute{v}_{p_1}(n)}, q_1^{\acute{v}_{q_1}(n)}\leq \frac{n}{i} \quad  \mbox{and} \quad  \frac{n}{j+1} < p_2^{\acute{v}_{p_2}(n)}, q_2^{\acute{v}_{q_2}(n)} \leq \frac{n}{j}
\end{equation}
for some $i, j= 1, 2, \dots, n$, then
\begin{equation*}
    a>n \quad \mbox{or} \quad b>n.
\end{equation*}
\end{lemma}
\begin{proof}
Let $\frac{a}{(a,n)}=p_1^{\alpha_1} p_2^{\alpha_2}$ and $\frac{b}{(b,n)}= q_1^{\beta_1} q_2^{\beta_2}$ such that $p_1,p_2,q_1,q_2$ are distinct prime numbers satisfying equation \eqref{inequality1}, where $\alpha_1,\alpha_2,\beta_1,\beta_2>0$ are positive integers. Consequently,
\begin{equation}\label{equality1}
     \left\lfloor\frac{n}{p_1^{\acute{v}_{p_1}(n)}}\right\rfloor=\left\lfloor\frac{n}{q_1^{\acute{v}_{q_1}(n)}}\right\rfloor
     \quad  \mbox{and} \quad  \left\lfloor\frac{n}{p_2^{\acute{v}_{p_2}(n)}}\right\rfloor=\left\lfloor\frac{n}{q_2^{\acute{v}_{q_2}(n)}}\right\rfloor.
\end{equation}
Since $\frac{a}{(a,n)}=p_1^{\alpha_1} p_2^{\alpha_2}$ and $\frac{b}{(b,n)}=q_1^{\beta_1}q_2^{\beta_2},$ therefore Lemmas \ref{rem1}, part $i.$, implies that
\begin{equation}\label{formula2}
    p_1^{\acute{v}_{p_1}(n)}p_2^{\acute{v}_{p_2}(n)} \ \Big| \ a\quad \mbox{and}\quad q_1^{\acute{v}_{q_1}(n)}q_2^{\acute{v}_{q_2}(n)} \ \Big| \ b.
\end{equation}
In the case $i=j,$ equation \eqref{inequality1} implies that
\begin{equation*}
    \left\lfloor\frac{n}{p_1^{\acute{v}_{p_1}(n)}}\right\rfloor=\left\lfloor\frac{n}{q_1^{\acute{v}_{q_1}(n)}}\right\rfloor
    =\left\lfloor\frac{n}{p_2^{\acute{v}_{p_2}(n)}}\right\rfloor=\left\lfloor\frac{n}{q_2^{\acute{v}_{q_2}(n)}}\right\rfloor.
 \end{equation*}
Therefore, using Lemma \ref{lem1.14} and equation \eqref{formula2}, we get
\begin{equation*}
    n<p_1^{\acute{v}_{p_1}(n)}p_2^{\acute{v}_{p_2}(n)}\leq a \quad \mbox{and}\quad n<q_1^{\acute{v}_{q_1}(n)}q_2^{\acute{v}_{q_2}(n)}\leq b.
\end{equation*}
If $i>j$, then equation  \eqref{formula2} implies that
\begin{equation*}
    p_1^{\acute{v}_{p_1}(n)}<p_2^{\acute{v}_{p_2}(n)}, q_2^{\acute{v}_{q_2}(n)} \quad\mbox{and}\quad q_1^{\acute{v}_{q_1}(n)}<p_2^{\acute{v}_{p_2}(n)}, q_2^{\acute{v}_{q_2}(n)}.
\end{equation*}
Therefore, using Lemma \ref{lem1.14} and  equations \eqref{equality1} and \eqref{formula2}, we get
\begin{equation*}
    n<p_1^{\acute{v}_{p_1}(n)}q_1^{\acute{v}_{q_1}(n)}< p_1^{\acute{v}_{p_1}(n)}p_2^{\acute{v}_{p_2}(n)}\leq a.
\end{equation*}
In a very similar way, one can prove that, if $i<j$, then
$$n<p_2^{\acute{v}_{p_2}(n)}q_2^{\acute{v}_{q_2}(n)}< q_1^{\acute{v}_{q_1}(n)}q_2^{\acute{v}_{q_2}(n)}\leq b.$$
Hence, the proof follows.
\end{proof}

\begin{lemma}\label{lem1.12}
 Let $a,b,t,n\in \mathbb{Z}^+$. If $\left\lfloor\frac{n}{a}\right\rfloor=\left\lfloor\frac{n}{b}\right\rfloor$, then $\left\lfloor\frac{n}{ta}\right\rfloor=\left\lfloor\frac{n}{tb}\right\rfloor$.
\end{lemma}

The proof of Lemma \ref{lem1.12} is a direct application of the fact $\left\lfloor\frac{x}{n}\right\rfloor=\left\lfloor\frac{\lfloor x\rfloor}{n}\right\rfloor$  for any $x\in\mathbb{R}$ and $n\in\mathbb{Z}^+,$ where $\mathbb{R}$ is the set of real numbers. Lemmas \ref{lem1.15} and \ref{lem1.12} will be used in the proof of Theorem \ref{thm eq-degree2}.

\section{Number of Edges of $D_n$}

In this section, we derive a formula for the number of edges of maximal Diophantine graphs $D_n$ using a lemma that characterizes nonadjacent vertices. In the remainder of this paper, let $D_n$ denote a maximal Diophantine graph with order $n$ that admits a Diophantine labelling $f$.

\begin{lemma}\label{lem1}
   For each two vertices $u,v$ in $D_n$, $uv\notin E(D_n)$ if and only if there exists a prime number $p$ such that $f(u), f(v)\in M_{p^{\acute{v}_{p}(n)}},$ where $M_{p^{\acute{v}_{p}(n)}}$ is the set of multiple of $p^{\acute{v}_{p}(n)}$ up to $n$.
\end{lemma}
\hspace{-0.5cm}\emph{Proof.} Let $u,v$ be two vertices in $D_n$ such that
\begin{equation*}
    \begin{split}
      uv\notin E(D_n)  &\quad\mbox{if and only if} \quad (f(u),f(v))\nmid n\\
                       &\quad\mbox{if and only if} \quad v_p\left(\big(f(u),f(v)\big)\right)>v_p(n) \ \mbox{for some prime} \ p\\
                       &\quad\mbox{if and only if} \quad \min\left\{v_p\left(f(u)\right), \ v_p\left(f(v)\right)\right\}\geq\acute{v}_{p}(n)\ \mbox{for some prime} \ p\\
                       &\quad\mbox{if and only if} \quad v_p(f(u))\geq \acute{v}_{p}(n) \ \mbox{and} \ v_p(f(v))\geq\acute{v}_{p}(n)\ \mbox{for some prime} \ p\\
                       &\quad\mbox{if and only if} \quad p^{\acute{v}_{p}(n)}\mid f(u)\ \mbox{and} \ p^{\acute{v}_{p}(n)}\mid f(v) \ \mbox{for some prime} \ p\\
                       &\quad\mbox{if and only if} \quad f(u), f(v)\in M_{p^{\acute{v}_{p}(n)}}\ \mbox{for some prime} \ p.\hspace{3.3cm}\square
    \end{split}
\end{equation*}

It is clear that from the Lemma \ref{lem1}, if such a prime number $p$ exists, then $p<\frac{n}{2}$. Any edge in $D_n$ represents a solvable linear Diophantine equation $ax+by=n$ and any non-edge in $D_n$ represents an unsolvable linear Diophantine equation $ax+by=n$, more precisely, using lemma \ref{lem1}, we have for each $u,v\in V(D_n)$, $uv\notin E(D_n)$ if and only if the linear equation $f(u)x+f(v)y=n$ is unsolvable.

\begin{theorem} \label{thm_V.I.}
The size of $D_n$ is given by
\begin{equation*}
\begin{split}
     |E(D_n)|=C(n,2)&-\sum_{1\leq i\leq \pi(n)} C\Bigg(\Bigg\lfloor\frac{n}{p_i^{\acute{v}_{p_i}(n)}}\Bigg\rfloor, 2\Bigg)
        +\sum_{1\leq i<j\leq \pi(n)}C\Bigg(\Bigg\lfloor\frac{n}{p_{i}^{\acute{v}_{p_i}}(n)p_{j}^{\acute{v}_{p_j}(n)}}\Bigg\rfloor, 2\Bigg)\\
        &-\dots + (-1)^{\pi(n)}C\left(\left\lfloor\frac{n}{\prod\limits_{1\leq i\leq \pi(n)} p_i^{\acute{v}_{p_i}(n)}}\right\rfloor, 2\right),
\end{split}
\end{equation*}
 where $p_i<n$, $i=1,2,\dots \pi(n)$ are distinct prime numbers, $\pi(n)$ is the number of prime numbers not exceeding $n$ and $C(n,r)$ is the binomial coefficient, defined as $C(n,r):=0$ if $n<r$ and $ C(n,r):=\binom{n}{r}$ if $n\geq r$.
\end{theorem}
\hspace{-0.5cm}\emph{Proof.}
Let $p_i<n$, $i=1,2, \dots, \pi(n)$ be distinct prime numbers. Then, using Lemma \ref{lem1}, we get for each two vertices $u,v$ in $D_n$, $uv\notin E(D_n)$ if and only if $f(u),f(v)\in M_{p_i^{\acute{v}_{p_i}(n)}}$ for some $i=1, 2, \dots, \pi(n)$. Define
\begin{equation*}
S_{p_i}:=\left\{uv\notin E(D_n):  f(u),f(v)\in M_{p_i^{\acute{v}_{p_i}(n)}}\right\},\quad i=1,2, \dots, \pi(n).
\end{equation*}
Thus, we have the following equation
\begin{equation*}
  \begin{split}
  |E(D_n)| &= |E(K_n)|-|E(\overline{D_n})|, \ \mbox{where $\overline{D_n}$ is the complement of $D_n$.}\\
           &= |E(K_n)|-\left|\bigcup\limits_{1\leq i\leq\pi(n)} S_{p_i}\right|.
  \end{split}
\end{equation*}
Using the inclusion-exclusion principle (see \cite{Rosen2}), we have
\begin{equation*}
\begin{split}
    |E(D_n)|
            &= C(n,2)-\sum_{1\leq i\leq\pi(n)}|S_{p_i}|+\sum_{1\leq i<j\leq\pi(n)}|S_{p_i}\cap S_{p_j}|
            -\dots+(-1)^{\pi(n)}\left|\bigcap\limits_{1\leq i\leq\pi(n)}S_{p_i}\right|.\\
            &= C(n,2)-\sum_{1\leq i\leq \pi(n)}C\left(\left|M_{p_i^{\acute{v}_{p_i}(n)}}\right|, 2\right)
            +\sum_{1\leq i<j\leq \pi(n)}C\left(\left|M_{p_i^{\acute{v}_{p_i}(n)}}\cap M_{p_j^{\acute{v}_{p_j}(n)}}\right| ,2\right)\\
            &\qquad\qquad\quad-\dots+(-1)^{\pi(n)}C\left(\left|\bigcap\limits_{1\leq i\leq\pi(n)}M_{p_i^{\acute{v}_{p_i}(n)}}\right|, 2\right).\\
            &= C(n,2)-\sum_{1\leq i\leq \pi(n)}C\left(\left\lfloor\frac{n}{p_i^{\acute{v}_{p_i}(n)}}\right\rfloor, 2\right)
            +\hspace{-0.1cm}\sum_{1\leq i<j\leq \pi(n)}C\left(\left\lfloor\frac{n}{p_{i}^{\acute{v}_{p_i}(n)}p_{j}^{\acute{v}_{p_j}(n)}}\right\rfloor, 2\right)\\
            &\qquad\qquad\quad-\dots+(-1)^{\pi(n)}C\left(\left\lfloor\frac{n}{\prod\limits_{1\leq i\leq \pi(n)} p_i^{\acute{v}_{p_i}(n)}}\right\rfloor, 2\right).\hspace{4.1cm}\square
  \end{split}
\end{equation*}

\section{Degrees of Vertices}

In this section, we proved that there are infinitely many Diophantine graphs which are not prime graphs using a characterization of formula for full degree vertices. Moreover, we give a formula for the degree of vertices. Necessary and sufficient conditions of equal degree vertices with specific labels and specific reduced labels are investigated.

\subsection{The Degree of a Vertex in $D_n$}

\begin{theorem}\label{lem2}
  For each vertex $u$ in $D_n$ and prime number $p$, $\deg(u)=n-1$ if and only if
   $$f(u)=p^{\acute{v}_p(n)}, \ \frac{n}{2}<f(u)<n, \quad \mbox{\textbf{or}} \quad f(u)\mid n,$$
  where the exclusive \textbf{or} will be typed in bold while the inclusive or is typed as usual.
\end{theorem}
\begin{proof}
Let $u\in V(D_n)$ with $\deg(u)=n-1$. By the way of contradiction, suppose
$$f(u)=p^{\acute{v}_p(n)}, \ \frac{n}{2}<f(u)<n, \quad \mbox{and} \quad f(u)\mid n,$$
or
$$\left(f(u)\neq p^{\acute{v}_p(n)} \quad \mbox{or} \quad 1\leq f(u)\leq\frac{n}{2}\right)\quad \mbox{and} \quad f(u)\nmid n,$$
where $p$ is a prime number. The first case is clearly a contradictory statement. The second case has two subcases as follows.
$$\left(f(u)\neq p^{\acute{v}_p(n)} \quad\mbox{and}\quad f(u)\nmid n\right) \quad\mbox{or}\quad \left( 1<f(u)<\frac{n}{2} \quad\mbox{and}\quad f(u)\nmid n\right),$$
In case of $f(u)\neq p^{\acute{v}_p(n)}$ and $f(u)\nmid n$, we get that there exists a prime number $p$ such that $p\mid f(u)$ and $v_p(f(u))>v_p(n).$ Then, we obtain $v_p\left(f(u)\right)\geq \acute{v}_p(n),$ so $p^{\acute{v}_p(n)}\mid f(u)$. Consequently, one can see that
$\left(f(u),p^{\acute{v}_p(n)}\right)=p^{\acute{v}_p(n)}\nmid n,$
which is a contradiction of the full degree of $u$. In the other case of $1<f(u)=p^{\acute{v}_p(n)}<\frac{n}{2}$, we have $(f(u),2p^{\acute{v}_p(n)})\nmid n,$
which contradicts the full degree of $u$. Hence, one can see that
$$f(u)=p^{\acute{v}_p(n)}, \ \frac{n}{2}<p^{\acute{v}_p(n)}<n \quad \mbox{\textbf{or}}\quad f(u)\mid n.$$

Conversely, in the first case of $f(u)\mid n$, we have that for each $v\in V(D_n)$, $(f(u),f(v))\mid n$. In the second case of $f(u)=p^{\acute{v}_p(n)}$ and $\frac{n}{2}<f(u)<n$, where $p$ is a prime number, we get that any proper multiple of $f(u)$ is greater than $n$. Since $f(u)=p^{\acute{v}_p(n)}$, therefore for all $v\in V(D_n)$, $(f(u), f(v))\mid p^{v_p(n)}$, which means for all $v\in V(D_n)$, $(f(u), f(v))\mid n$. In both cases, we get $\deg(u)=n-1$.
\end{proof}

Similar results can be deduced in the case of a maximal prime graphs $R_n$ for Lemma \ref{lem1} and Theorem \ref{lem2}. In light of Lemma \ref{lem1}, Theorem \ref{lem2}, and the following theorem \ref{C4} due to M. A. Seoud and M. Z. Youssef, we have the following results.
\begin{theorem}\cite{Seoud1}, \cite{thesis}\label{C4}
  If $G$ is a simple graph of $n$ vertices which has more than
  $$\pi\left(n\right)+\pi\left(\frac{n}{2}\right)+1$$
  vertices of degree $n-1$, then $G$ is not prime, where
  $\pi(n)$ is the number of prime numbers not exceeding $n$.
\end{theorem}

\begin{corollary}
  There are infinitely many Diophantine graphs which are not prime graphs.
\end{corollary}
\begin{proof}
Let $p$ be a prime number and an integer $k>1$. Maximal Diophantine graphs $D_{p^k}$ with $p^k$ vertices can not be prime graphs because the number of full degree vertices in $R_{p^k}$ that is equal to
$$\pi\left(p^k\right)+\pi\left(\frac{p^k}{2}\right)+1, \ \mbox{(by appling Theorem \ref{C4})},$$
and it is less than the number of full degree vertices in $D_{p^k}$ which is at least
$$\pi\left(p^k\right)+\pi\left(\frac{p^k}{2}\right)+k+1, \ \mbox{(by the aid of Theorem \ref{lem2})}.$$
Hence, the proof follows.
\end{proof}

\begin{lemma}\label{lem3}
   Let $u$ be a vertex in $D_n$. If $f^*(u)=p^k$ for some prime number $p$ and $k\in\mathbb{N}$, then
\begin{equation*}
\deg(u)=\left\{
           \begin{array}{ll}
             n-1,                                                      & \hbox{$k=0.$} \\
             n-\Big\lfloor\frac{n}{p^{\acute{v}_p(n)}}\Big\rfloor,     & \hbox{$k>0.$}
           \end{array}
       \right.
\end{equation*}
\end{lemma}
\begin{proof}
Let $f^*(u)=p^k$, where $u\in V(D_n)$, $p$ is a prime number and $k\in\mathbb{N}$. In the case of $k=0$, we have $f^*(u)=1$. Since $f^*(u)=\frac{f(u)}{(f(u),n)}$, we have $f(u)\mid n$. Therefore, using Theorem \ref{lem2}, $\deg(u)=n-1$. In the other case of $k>0$, we have
$$p\mid f^*(u) \quad\mbox{if and only if}\quad p^{\acute{v}_p(n)}\mid f(u), \quad \mbox{(by applying Lemmas \ref{rem1}, part $i.$)}.$$
Let $s\in V(D_n)-\{u\}$ such that
$$su\notin E(D_n)\quad\mbox{if and only if}\quad f(s), f(u)\in M_{p^{\acute{v}_p(n)}},\quad \mbox{(by the aid of Lemma \ref{lem1})}.$$
Therefore,
$\deg(u)=n-1-\left(\left|M_{p^{\acute{v}_p(n)}}\right|-1\right)=n-\left\lfloor\frac{n}{p^{\acute{v}_p(n)}}\right\rfloor.$
\end{proof}

In a very similar way of Lemma \ref{lem3}, using the inclusion-exclusion principle and the identity $\sum\limits_{i=0}^{r}(-1)^r\binom{r}{i}=0$, one can prove the following theorem.\\

\begin{theorem}\label{deg(u)}
   Let $u$ be a vertex in $D_n$. If $f^*(u)=\prod\limits_{i=1}^{r}p_i^{k_i}$, where $p_i$, $i=1,2,\dots,r$ are distinct prime numbers and $k_i\in\mathbb{N}$, $i=1,2,\dots,r$, then
\begin{equation*}
\deg(u)=\left\{\begin{array}{ll}
               n-1,                                                                                                            &\\
               \hspace{10cm}\hbox{if $f^*(u)=1.$} \\
               n-\sum\limits_{1\leq i\leq r}\left\lfloor\frac{n}{p_i^{\acute{v}_{p_i}(n)}}\right\rfloor
                 +\sum\limits_{1\leq i.j\leq r}\left\lfloor\frac{n}{p_i^{\acute{v}_{p_i}(n)}p_j^{\acute{v}_{p_j}(n)}}\right\rfloor
                 -\dots +(-1)^{r}\left\lfloor\frac{n}{\prod\limits_{1\leq i\leq r}p_i^{\acute{v}_{p_i}(n)}}\right\rfloor,       &\\
               \hspace{10cm}\hbox{if $f^*(u)>1.$}
              \end{array}
       \right.
\end{equation*}
\end{theorem}

 Results similar to Theorems \ref{thm_V.I.}, \ref{deg(u)} can be obtained in a maximal prime graph $R_n$, by putting $v_p(n)=0$.

\subsection{Vertices of Equal Degrees $<n-1$}

\begin{lemma}\label{lem4}
  Let $u,v$ be two vertices in $D_n$. If $f(u)\mid f(v)$, then $N(u)-\{v\}\supseteq N(v)-\{u\}$.
\end{lemma}
\begin{proof}
The proof is straightforward.
\end{proof}

\begin{lemma}\label{lem5}
  Let $u,v$ be two vertices in $D_n$ such that $f(u)\mid f(v)$. If $\deg(u)=\deg(v)$, then $N(u)-\{v\}=N(v)-\{u\}$.
\end{lemma}
\begin{proof}
Let $u,v$ be two vertices in $D_n$ such that
$f(u)\mid f(v) \ \mbox{and} \ \deg(u)=\deg(v).$
Let $s\in V(D_n)-\{u, v\}$ such that $sv\in E(D_n)$. Then, using Lemma \ref{lem4}, $su\in E(D_n)$. Hence, $N(u)-\{v\}\supseteq N(v)-\{u\}$. On the other hand, let $s\in V(D_n)-\{u, v\}$ such that $su\in E(D_n)$. Suppose by contrary that $sv\notin E(D_n)$. Since $f(u)\mid f(v)$ and using Lemma \ref{lem4}, therefore $N(u)-\{v\}\supset N(v)-\{u\}$. Then we have $\deg(u)> \deg(v)$, which contradicts the assumption of $\deg(u)=\deg(v)$. Therefore, $sv\in E(D_n)$. Consequently, $N(u)-\{v\}\subseteq N(v)-\{u\}$. Hence, $N(u)-\{v\}=N(v)-\{u\}$.
\end{proof}

It is clear that if $N(u)-\{v\}=N(v)-\{u\}$, then $\deg(u)=\deg(v)$. Thus, under the assumption that $f(u)\mid f(v)$, we have
$N(u)-\{v\}=N(v)-\{u\}$ if and only if $\deg(u)=\deg(v).$
Furthermore, the following example shows that the condition $f(u)\mid f(v)$ is necessary for Lemma \ref{lem5}.

\begin{example}\label{ex1}
  In $D_{20}$, let $f(u)=7$ and $f(v)=8$, where $u,v\in V(D_{20})$. Then $N(u)=V(D_{20})-\{u_1, u\}$, where $u_1\in V(D_{20})$ such that $f(u_1)=14$ and $N(v)=V(D_{20})-\{v_1, v\}$, where $v_1\in V(D_{20})$ such that $f(v_1)=16$. Then we have $\deg(u)=\deg(v)$, while $f(u)\nmid f(v)$, but $N(u)-\{v\}\neq N(v)-\{u\}$.
\end{example}

\begin{lemma}\label{lem6}
  Let $u$ be a vertex in $D_n$. If $deg(u)<n-1$, then there exists a prime number $p$ such that $p^{\acute{v}_p(n)}\mid f(u)$.
\end{lemma}
\begin{proof}
Let $u$ be a vertex in $D_n$ with $deg(u)<n-1$. Then there exists a vertex $s\in V(D_n)$ such that $su\notin E(D_n)$. Therefore, using Lemma \ref{lem1}, there exists a prime number $p$ such that $f(s),f(u)\in M_{p^{\acute{v}_p(n)}}$. Hence, there exists a prime number $p$ such that $p^{\acute{v}_p(n)}\mid f(u)$.
\end{proof}

\begin{theorem}\label{thm eq-degree1}
  Let $f(u)=p^t$ and $f(v)=q^k$ for some $t,k\in\mathbb{Z^+}$, where $u,v\in V(D_n)$ with $\deg(u)<n-1$, $\deg(v)<n-1$ and $p,q$ are prime numbers. Then we have
\begin{equation*}
 \deg(u)=\deg(v) \quad \mbox{if and only if} \quad \acute{v}_p(n)\leq t, \ \acute{v}_q(n)\leq k \quad \mbox{and} \quad \left\lfloor\frac{n}{p^{\acute{v}_p(n)}}\right\rfloor =\left\lfloor\frac{n}{q^{\acute{v}_q(n)}}\right\rfloor.
\end{equation*}
\end{theorem}
\begin{proof}
Let $f(u)=p^t$ and $f(v)=q^k$ for some $t,k\in\mathbb{Z^+}$, where $p,q$ are prime numbers and $u,v\in V(D_n)$ with $deg(u)=deg(v)$, $deg(u)<n-1$ and $deg(v)<n-1$. Therefore, using Lemma \ref{lem6}, we have
$$p^{\acute{v}_p(n)}\mid f(u)=p^t\quad \mbox{and}\quad q^{\acute{v}_q(n)}\mid f(v)=q^k.$$
Consequently, $\acute{v}_p(n)\leq t$ and $\acute{v}_q(n)\leq k$. Using Lemma \ref{lem3}, we get
$$deg(u)=n-\left\lfloor\frac{n}{p^{\acute{v}_p(n)}}\right\rfloor\quad \mbox{and}\quad \deg(v)=n-\left\lfloor\frac{n}{q^{\acute{v}_q(n)}}\right\rfloor.$$
Hence, we obtain the following equation
$\left\lfloor\frac{n}{p^{\acute{v}_p(n)}}\right\rfloor =\left\lfloor\frac{n}{q^{\acute{v}_q(n)}}\right\rfloor.$

Conversely, let $\left\lfloor\frac{n}{p^{\acute{v}_p(n)}}\right\rfloor =\left\lfloor\frac{n}{q^{\acute{v}_q(n)}}\right\rfloor$ and $\acute{v}_p(n)\leq t$, $\acute{v}_q(n)\leq k$. Then we get $$p^{\acute{v}_p(n)}\mid f(u) \quad \mbox{and} \quad q^{\acute{v}_q(n)}\mid f(v).$$
Thus, using Lemma \ref{lem3}, we have the following two equations $$\deg(u)=n-\left\lfloor\frac{n}{p^{\acute{v}_p(n)}}\right\rfloor\quad\mbox{and}\quad\deg(v)=n-\left\lfloor\frac{n}{q^{\acute{v}_q(n)}}\right\rfloor.$$ Hence, the proof follows.
\end{proof}

\begin{corollary}\label{thm eq-degree}
  Let $f(u)=p^t$ and $f(v)=p^k$ for some $t,k\in\mathbb{Z^+}$, where $u,v\in V(D_n)$ with $\deg(u)<n-1$, $\deg(v)<n-1$ and $p$ is a prime number. Then we have
 \begin{equation*}
    \deg(u)=\deg(v) \quad \mbox{if and only if} \quad \acute{v}_p(n)\leq \min\{t,k\}.
 \end{equation*}
\end{corollary}

\begin{theorem}\label{thm eq-degree3}
  Let $u,v$ be two vertices in $D_n$ such that $f(v)=tf(u)$  for some $t\in\mathbb{Z^+}$ and $f(v)$ is not a prime power. Then we have $\deg(u)=\deg(v)$ if and only if for every prime number $p\mid t$,
  \begin{equation*}
        \acute{v}_p(n)\leq v_p(f(u)) \quad \mbox{\textbf{or}} \quad \acute{v}_p(n)>v_p(f(v)).
  \end{equation*}
\end{theorem}
\begin{proof}
Let $u,v$ be two vertices in $D_n$ such that $f(v)=tf(u)$ for some $t\in\mathbb{Z^+}$ and $f(v)$ is not a prime power. Proof by contraposition, suppose that there exists a prime number $p_0\mid t$ such that
\begin{equation*}
   \acute{v}_{p_0}(n)\leq v_{p_0}(f(u)) \quad \mbox{and} \quad \acute{v}_{p_0}(n)>v_{p_0}(f(v)),
\end{equation*}
or
\begin{equation*}
   \acute{v}_{p_0}(n)>v_{p_0}(f(u)) \quad \mbox{and} \quad \acute{v}_{p_0}(n)\leq v_{p_0}(f(v)).
\end{equation*}
The first case contradicts assumptions $f(u)\mid f(v)$. Therefore, let us consider the second case, which is
\begin{equation}\label{inequality2}
   v_{p_0}(f(u))<\acute{v}_{p_0}(n)\leq v_{p_0}(f(v)).
\end{equation}
Let a vertex $s\in V(D_n)$ such that $f(s)=p_0^{\acute{v}_{p_0}(n)}$. Then, using equation \eqref{inequality2}, we have
$p_0^{\acute{v}_{p_0}(n)}\nmid p_0^{\acute{v}_{p_0}(f(u))}.$
Since $p_0^{\acute{v}_{p_0}(f(u))}\mid f(u)$, we get $f(s)=p_0^{\acute{v}_{p_0}(n)}\nmid f(u)$. Consequently, $s\neq u$. In addition,  $s\neq v$, since $f(v)$ is not a prime power. Moreover, using equation \eqref{inequality2}, we get
$$(f(u), f(s))\mid n \quad \mbox{while} \quad (f(v), f(s))\nmid n.$$
Thus,
$su\in E(D_n) \ \mbox{while} \ sv\notin E(D_n).$
Since $f(u)\mid f(v)$ and using Lemma \ref{lem4}, we have $N(u)-\{v\}\supset N(v)-\{u\}$. Hence, $\deg(u)\neq\deg(v)$.

Conversely, suppose that, by contrapositive, $deg(u)\neq deg(v)$. Then, using Lemma \ref{lem4}, $N(u)-\{v\}\supset N(v)-\{u\}$. So, there exists a vertex $s\in V(D_n)-\{u,v\}$ such that
$s\in N(u)-\{v\} \ \mbox{and} \ s\notin N(v)-\{u\}.$
Consequently,
\begin{equation}\label{gcd}
    (f(u), f(s))\mid n \quad\mbox{and}\quad (f(v), f(s))\nmid n.
\end{equation}
Then, using Lemma \ref{lem1}, there exists a prime number $p_0$ such that
\begin{equation}\label{multip}
  p_0^{\acute{v}_{p_0}(n)}\nmid f(u) \quad\mbox{and}\quad f(v), f(s)\in M_{p_0^{\acute{v}_{p_0}(n)}}.
\end{equation}
Let $f(s)=p_0^{\acute{v}_{p_0}(n)}$. Then, using equation \eqref{multip}, we have $s\neq u$. Since $f(v)=t f(u)$, we get $p_0\mid t$ and $s\neq v$. Therefore, using equation \eqref{gcd}, we have
\begin{equation}\label{p-adic}
v_{p_0}\left((f(u), f(s)) \right)\leq v_{p_0}(n) \quad\mbox{and}\quad v_{p_0}\left((f(v), f(s)) \right)>v_{p_0}(n).
\end{equation}
Hence, using equation \eqref{p-adic}, we get
$$v_{p_0}(f(u))<\acute{v}_{p_0}(n) \quad\mbox{and}\quad v_{p_0}(f(v))\geq\acute{v}_{p_0}(n)$$
for some prime number $p_0\mid t$.
\end{proof}

\subsection{More Necessary and Sufficient Conditions for Equality of Degrees} The following theorem is a generalization of one direction of Theorem \ref{thm eq-degree1}.

\begin{theorem}\label{thm eq-degree2}
  Let $u,v$ be two vertices in $D_n$ with $\deg(u)<n-1$ and $\deg(v)<n-1$. If $\omega\big(f^*(u)\big)=\omega\big(f^*(v)\big)$ and for every prime number $p\mid f^*(u)$, there exists a prime number $q\mid f^*(v)$ such that $\left\lfloor\frac{n}{p^{\acute{v}_p(n)}}\right\rfloor=\left\lfloor\frac{n}{q^{\acute{v}_q(n)}}\right\rfloor$, then $\deg(u)=\deg(v)$, where $w(n)$ is the number of distinct prime factors of $n$.
\end{theorem}
\begin{proof}
Let $u,v$ be two vertices in $D_n$ such that $\deg(u)<n-1$, $\deg(v)<n-1$  and
\begin{equation}\label{w=w}
  \omega\big(f^*(u)\big)=\omega\big(f^*(v)\big).
\end{equation}
Then, using Lemma \ref{lem6}, there exists a prime number $p$ such that $p^{\acute{v}_p(n)}\mid f(u).$ Therefore, using Lemma \ref{rem1}, part i., there exists a prime number $p$ such that $p\mid f^*(u)$. Consequently, It follows from our assumption that there exists a prime number $q$ such that $q\mid f^*(v)$ and
\begin{equation}\label{equalfloor}
\left\lfloor\frac{n}{p^{\acute{v}_p(n)}}\right\rfloor=\left\lfloor\frac{n}{q^{\acute{v}_q(n)}}\right\rfloor.
\end{equation}
Moreover, $\omega\big(f^*(u)\big)=\omega\big(f^*(v)\big)>0.$ Since $p^{\acute{v}_p(n)}\mid f(u)$ and $q^{\acute{v}_q(n)}\mid f(v)$ and using Theorem \ref{FTA}, therefore there exist $\grave{u},\grave{v}\in V(D_n)$ such that
\begin{equation*}
f(\grave{u}):=\prod\limits_{p\mid f^*(u)}p^{\acute{v}_p(n)}\ \bigg| \ f(u) \quad \mbox{and} \quad f(\grave{v}):=\prod\limits_{q\mid f^*(v)}q^{\acute{v}_q(n)} \ \bigg| \ f(v).
\end{equation*}
By definition of $f(\grave{u})$ and $f(\grave{v})$, we have that the reduced labels $f^*(u),f^*(\grave{u})$ possess identical prime divisors and also the reduced labels $f^*(v),f^*(\grave{v})$ have the same prime factors. Therefore, using equation \eqref{w=w},
 \begin{equation}\label{equality2}
\omega\big(f^*(\grave{u})\big)=\omega\big(f^*(u)\big) =\omega\big(f^*(v)\big) =\omega\big(f^*(\grave{v})\big).
\end{equation}
Furthermore, for every prime number $p$, we have
 $$p^{\acute{v}_{p}(n)}\mid f(u) \ \mbox{if and only if} \ p^{\acute{v}_{p}(n)}\mid f(\grave{u}).$$
Therefore,
\begin{equation*}
\bigcup\limits_{p|f^*(u)}M_{p^{\acute{v}_{p}(n)}}=\bigcup\limits_{p|f^*(\grave{u})}M_{p^{\acute{v}_{p}(n)}}.
\end{equation*}
Consequently, using Lemma \ref{lem1}, for each vertex $s\in V(D_n)-\{u, \ \grave{u}\}$, we have
$$su\in E(D_n) \quad\mbox{if and only if}\quad s\grave{u}\in E(D_n).$$
Hence,
\begin{equation}\label{equal-deg1}
    \deg(u)=\deg(\grave{u}).
\end{equation}
Similarly, we get
\begin{equation}\label{equal-deg2}
    \deg(v)=\deg(\grave{v}).
\end{equation}
Let $d=(f(\grave{u}), f(\grave{v}))$. Thus, Lemmas \ref{rem1}, part $ii.$, implies that
\begin{equation}\label{divide1}
\frac{d}{(d, n)} \ \bigg| \ f^*(\grave{u}),  f^*(\grave{v}),
\end{equation}
and hence
$\omega\left(\frac{d}{(d, n)}\right)\leq\omega\left(f^*(\grave{u})\right)=\omega\left(f^*(\grave{v})\right).$
Let us consider the following cases.\\
Case 1: Let $\omega\left(f^*(\grave{u})\right)=\omega\left(f^*(\grave{v})\right)=\omega\left(\frac{d}{(d, n)}\right).$ Then, using equation \eqref{divide1}, $\frac{d}{(d, n)}$, $f^*(u)$ and $f^*(v)$ have the same prime factors. A similar argument used to prove equation \eqref{equal-deg1} implies that
$\deg(\grave{u}) = \deg(\grave{v}).$
Hence, using equations \eqref{equal-deg1} and \eqref{equal-deg2}, we get
$\deg(u)=\deg(v).$\\
Case 2: Let $\omega\big(f^*(\grave{u})\big)=\omega\big(f^*(\grave{v})\big)=\omega\left(\frac{d}{(d, n)}\right)+1=r.$ Then we get $f^*(\grave{u})$ and $f^*(\grave{v})$ differ in exactly one prime factor. Therefore, using Theorem \ref{FTA}, we obtain the following two equations
\begin{equation*}
f(\grave{u})=\prod\limits_{i=1}^{r-1}h_i^{\acute{v}_{h_i}(n)}p^{\acute{v}_p(n)} \quad \mbox{and} \quad f(\grave{v})=\prod\limits_{i=1}^{r-1}h_i^{\acute{v}_{h_i}(n)}q^{\acute{v}_q(n)},
\end{equation*}
where $h_i, p$ and $q$ are distinct prime numbers such that $h_i\mid d, p\nmid d$ and $q\nmid d$ for all $i=1, 2, \dots, r-1$. Then, using Theorem $\ref{deg(u)}$, we have the following two equations
\begin{equation}\label{eqn1}
\begin{split}
  \deg(\grave{u})=n&-\sum_{1\leq i\leq r-1}\left\lfloor\frac{n}{h_i^{\acute{v}_{h_i}(n)}}\right\rfloor-\left\lfloor\frac{n}{p^{\acute{v}_{p}(n)}}\right\rfloor
  +\sum_{1\leq i<j\leq r-1}\left\lfloor\frac{n}{h_i^{\acute{v}_{h_i}(n)}h_j^{\acute{v}_{h_j}(n)}}\right\rfloor\\
  &+\sum_{1\leq i\leq r-1} \left\lfloor\frac{n}{h_i^{\acute{v}_{h_i}(n)}p^{\acute{v}_{p}(n)}}\right\rfloor
  -\dots+(-1)^{r}\left\lfloor\frac{n}{\prod\limits_{1\leq i\leq r-1} h_i^{\acute{v}_{h_i}(n)}p^{\acute{v}_{p}(n)}}\right\rfloor,
\end{split}
\end{equation}
and
\begin{equation}\label{eqn2}
\begin{split}
  \deg(\grave{v})=n&-\sum_{1\leq i\leq r-1}\left\lfloor\frac{n}{h_i^{\acute{v}_{h_i}(n)}}\right\rfloor-\left\lfloor\frac{n}{q^{\acute{v}_{q}(n)}}\right\rfloor
  +\sum_{1\leq i<j\leq r-1}\left\lfloor\frac{n}{h_i^{\acute{v}_{h_i}(n)}h_j^{\acute{v}_{h_j}(n)}}\right\rfloor\\
  &+\sum_{1\leq i\leq r-1} \left\lfloor\frac{n}{h_i^{\acute{v}_{h_i}(n)}q^{\acute{v}_{q}(n)}}\right\rfloor
  -\dots+(-1)^{r}\left\lfloor\frac{n}{\prod\limits_{1\leq i\leq r-1} h_i^{\acute{v}_{h_i}(n)}q^{\acute{v}_{q}(n)}}\right\rfloor.
\end{split}
\end{equation}
Using Lemma \ref{lem1.12} and equations (\ref{equalfloor}), (\ref{eqn1}) and (\ref{eqn2}), we get
$\deg(\grave{u})=\deg(\grave{v}),$
and hence
$\deg(u)=\deg(v).$\\
Case 3: Let $\omega\big(f^*(\grave{u})\big)=\omega\big(f^*(\grave{v})\big)=\omega\left(\frac{d}{(d, n)}\right)+2=r.$ Then $f^*(\grave{u})$ and $f^*(\grave{v})$ differ in exactly two prime factors. Therefore, using Theorem \ref{FTA}, we have the following two equations
\begin{equation*}
f(\grave{u})=\prod\limits_{i=1}^{r-2}h_i^{\acute{v}_{h_i}(n)}p_1^{\acute{v}_{p_1}(n)}p_2^{\acute{v}_{p_2}(n)} \quad \mbox{and} \quad f(\grave{v})=\prod\limits_{i=1}^{r-2}h_i^{\acute{v}_{h_i}(n)}q_1^{\acute{v}_{q_1}(n)}q_2^{\acute{v}_{q_1}(n)},
\end{equation*}
where $h_i, p_1, p_2, q_1$ and $q_2$ are distinct prime numbers satisfying $h_i\mid d$, $p_1\nmid d$, $p_2\nmid d$, $q_1\nmid d$ and $q_2\nmid d$ for all $i=1, 2, \dots, r-2$ such that
$$\left\lfloor\frac{n}{p_1^{\acute{v}_{p_1}(n)}}\right\rfloor =\left\lfloor\frac{n}{q_1^{\acute{v}_{q_1}(n)}}\right\rfloor\quad\mbox{and} \quad \left\lfloor\frac{n}{p_2^{\acute{v}_{p_2}(n)}}\right\rfloor =\left\lfloor\frac{n}{q_2^{\acute{v}_{q_2}(n)}}\right\rfloor.$$
Define
$a:=p_1^{\acute{v}_{p_1}(n)}p_2^{\acute{v}_{p_2}(n)} \ \mbox{and} \  b:=q_1^{\acute{v}_{q_1}(n)}q_2^{\acute{v}_{q_1}(n)}.$
Since
$\frac{a}{(a,n)}=p_1p_2 \ \mbox{and} \ \frac{b}{(b,n)}=q_1q_2,$
therefore Lemma \ref{lem1.15} implies that $a>n \ \mbox{or} \   b>n.$ Hence,
$f(\grave{u})\geq a>n \ \mbox{or} \ f(\grave{v})\geq b>n,$
which is absurd.\\
The cases $\omega\big(f^*(\grave{u})\big)=\omega\big(f^*(\grave{v})\big)=\omega\left(\frac{d}{(d, n)}\right)+i$, for every $i>3$ can be treated in a very similar way of case $3$ which are absurd. Hence the proof follows.
\end{proof}

\begin{example}
  In $D_{20}$, let $f(u)=14$ and $f(v)=16$, where $u,v\in V(D_{20})$, so $f^*(u)=7$ and $f^*(v)=8$. Then, applying Theorem \ref{thm eq-degree2}, we have $\omega\left(7\right)=\omega\left(8\right)=1$ and $\left\lfloor\frac{20}{7}\right\rfloor=\left\lfloor\frac{20}{8}\right\rfloor=2.$ Hence, $\deg(u)=\deg(v)=18$, (by applying Lemma \ref{lem3}).
\end{example}

The converse of Theorem \ref{thm eq-degree2} is not true. For instance, in $D_{23}$, let $f(u)=10$, $f(v)=14$, where $u,v\in V(D_{23})$. Then, by Theorem \ref{deg(u)}, we have $deg(u)=13=deg(v)$. However, it can be verified that $\left\lfloor\frac{23}{5}\right\rfloor\neq\left\lfloor\frac{23}{7}\right\rfloor$.

\begin{corollary}\label{thm_eq-deq1}
  Let $u,v$ be two vertices in $D_n$ with $f^*(u)>1$, $f^*(v)>1$. If $f^*(u),f^*(v)$ have the same prime factors, then $\deg(u)=\deg(v)$.
\end{corollary}

\begin{corollary}\label{cor1}
  Let $u,v$ be two vertices in $D_n$ such that $f(v)=tf(u)$  for some $t\in \mathbb{Z^+}$. If $(t, f(u))=1$ and $t\mid n$, then $\deg(u)=\deg(v)$.
\end{corollary}
\begin{proof}
Let $u,v$ be two vertices in $D_n$ such that $f(v)=tf(u)$  for some $t\in \mathbb{Z^+}$, $(t, f(u))=1$ and $t\mid n$. The reduced labels $f^*(v),f^*(u)$ satisfy the following equalities.
  $$f^*(v)=\frac{f(v)}{(f(v),n)}=\frac{tf(u)}{(tf(u),n)}=\frac{tf(u)}{t(f(u),n)}=\frac{f(u)}{(f(u),n)}=f^*(u).$$
Therefore,  $f^*(u),f^*(v)$ have the same prime factors \textbf{or} $f^*(u)=f^*(v)=1$. Hence, using Corollary \ref{thm_eq-deq1} and Theorem \ref{lem2}, the proof follows.
\end{proof}

The following two examples show that the converse of Corollary \ref{thm_eq-deq1} is not true.

\begin{example}
  Let $u,v$ be two different vertices in $D_n$ such that $f(u)\mid f(v)$ and $f(v)=p^{\acute{v}_{p}(n)}$, $\frac{n}{2}<f(v)<n$, where $p$ is a prime number. Then, using Theorem \ref{lem2}, we have
  $\deg(u)=n-1=\deg(v).$
  However, $f^*(u)=\frac{f(u)}{(f(u),n)}=1$ and $f^*(v)=\frac{p^{\acute{v}_{p}(n)}}{\left(p^{\acute{v}_{p}(n)},n\right)}=p$ have not the same prime factors.
\end{example}

\begin{example}
  Let $u,v$ be two vertices in $D_n$ such that $f(u)=2p^{\acute{v}_{p}(n)}$ and $f(v)=2q^{\acute{v}_{q}(n)}$, $\frac{n}{3}<p^{\acute{v}_{p}(n)}, q^{\acute{v}_{q}(n)}<\frac{n}{2}$, where $p,q$ are distinguish prime numbers, this means $f(u)\nmid f(v)$. Then, using Lemma \ref{lem3} and Lemmas \ref{rem1}, part $i.$, we have
  $$\deg(u)=n-\left\lfloor\frac{n}{p^{\acute{v}_p(n)}}\right\rfloor= n-2 =n-\left\lfloor\frac{n}{q^{\acute{v}_q(n)}}\right\rfloor =\deg(v).$$
 However, $f^*(u)=p$ and $f^*(v)=q$ have distinct prime factors.
\end{example}

The last two examples motivate the following partial converse of Corollary \ref{thm_eq-deq1}.

\begin{theorem}\label{thm_eq-deq2}
  Let $u,v$ be two vertices in $D_n$ such that $f(u)\mid f(v)$, $f(v)$ is not a prime power and $f^*(u)>1$. If $\deg(u)=\deg(v)$, then $f^*(u),f^*(v)$ have the same prime factors.
\end{theorem}
\begin{proof}
Let $u,v$ be two vertices in $D_n$ such that $f(u)\mid f(v)$, $f(v)$ is not a prime power, $\deg(u)=\deg(v)$ and $f^*(u)>1$. Then, using Lemmas \ref{rem1}, part $ii.$, we have $f^*(u)\mid f^*(v)$. Consequently, $f^*(v)>1$ and also any prime factor of $f^*(u)$ is a prime factor of $f^*(v)$. Let there exists a prime number $p$ such that $p\mid f^*(v)$. Then, using Lemmas \ref{rem1}, part $i.$, $p^{\acute{v}_{p}(n)}\mid f(v)$, i.e., $f(v)\in M_{p^{\acute{v}_{p}(n)}}$. Let a vertex $s\in V(D_n)$ such that $f(s)=p^{\acute{v}_{p}(n)}\in M_{p^{\acute{v}_{p}(n)}}.$ Since $f(v)$ is divisible by at least two distinct prime numbers, we get $s\neq v$. Using Lemma \ref{lem1}, we have $(f(v), f(s))\nmid n$. Therefore, $s\notin N(v)-\{u\}$. Consequently, using Lemma \ref{lem5}, we obtain $s\notin N(u)-\{v\}$. Then we get $\left(f(u),f(s)\right)\nmid n$. Thus, $f(s)=p^{\acute{v}_{p}(n)}\mid f(u)$. Consequently, using Lemmas \ref{rem1}, part $i.$, $p\mid f^*(u)$. Therefore, for every prime number $p$,
$$p\mid f^*(u)\quad \mbox{if and only if}\quad p\mid f^*(v).$$
Hence, the proof follows.
\end{proof}

\section{Conclusions}
 This manuscript introduces the concept of Diophantine labelling and establishes a connection between Diophantine graphs and prime graphs. We study further define maximal Diophantine graphs and derives explicit formulas for their edge count and vertex degrees. Additionally, we provide necessary and sufficient conditions for the equality of vertex degrees. We also utilize number-theoretic concepts, such as p-adic valuation, to analyze vertex labels and the properties of Diophantine graphs. Although Diophantine labelling is still primarily theoretical, we think that Diophantine labelling sits at an intersection of graph theory, number theory and combinatorics. Furthermore, exploring different variants of Diophantine labelling can further enrich the theory and uncover new connections between these mathematical domains. As research advances, we expect that Diophantine labelling has more concrete applications in the future. Moreover, the algorithmic approach is promising in many directions, including searching for an efficient algorithm for Diophantine graph and finding maximal Diophantine graphs.
%

\subsection*{Acknowledgment}
The author would like to extend their gratitude to the referees for their valuable feedback, insightful comments, and constructive suggestions, which significantly contributed to improving the quality and clarity of this work.



\end{document}